\documentclass[12pt]{amsart}
\usepackage{amsthm}
\usepackage{amsmath}
\usepackage{amssymb,verbatim}
\usepackage{tikz-cd}
\usepackage{todonotes}

\theoremstyle{plain}
\newtheorem{theorem}{Theorem}[section]
\newtheorem*{nonumbertheorem}{Theorem}

\newtheorem{corollary}[theorem]{Corollary}

\newtheorem{lemma}[theorem]{Lemma}

\newtheorem{proposition}[theorem]{Proposition}

\theoremstyle{definition}
\newtheorem{definition}[theorem]{Definition}

\newtheorem*{nonumberquestion}{Question}
 
\newtheorem{notation}[theorem]{Notation} 
\newtheorem{remark}[theorem]{Remark}
\newtheorem{example}[theorem]{Example}

\newtheorem{maintheorem}{Theorem}

\newtheorem{maincorollary}[maintheorem]{Corollary}

\numberwithin{table}{section}

\newcommand{\Z}{\mathbb{Z}}
\newcommand{\R}{\mathbb{R}}
\newcommand{\s}{\mathbb{S}}
\newcommand{\fol}{\mathcal{F}}
\newcommand{\q}{Q_1}

\DeclareMathOperator{\im}{im}

\title{On the topology of leaves of singular Riemannian foliations}
\author{Marco Radeschi}
\address[M. Radeschi]{Department of Mathematics, University of Notre Dame, Notre Dame, IN 46556}
\email{mradesch@nd.edu}
\author{Elahe Khalili Samani}
\address[E. Khalili Samani]{Department of Mathematics, University of Notre Dame, Notre Dame, IN 46556}
\email{ekhalili@nd.edu}
\begin{document}

\begin{abstract}
In this paper, we establish a number of results about the topology of the leaves of a closed singular Riemannian foliation $(M,\fol)$. If $M$ is simply connected, we prove that the leaves are finitely covered by nilpotent spaces, and characterize the fundamental group of the generic leaves. If $M$ has virtually nilpotent fundamental group, we prove that the leaves have virtually nilpotent fundamental group as well.
\end{abstract}

\maketitle

\section{Introduction}

The study of isometric group actions on Riemannian manifolds has seen a number of important applications in Riemannian geometry.

Many of them fall under the umbrella of the so-called \emph{Grove's program}, whose goal is to study the properties of Riemannian manifolds with non-negative (or even almost non-negative) sectional curvature in the presence of symmetry. This program has been extremely fruitful both in producing new examples of manifolds with non-negative sectional curvature, and in proving important conjectures in the area when some symmetry is added (cf. \cite{KWW21}, \cite{GKS20}, \cite{FGT17}, \cite{GW14}, \cite{GZ00}, \cite{GVZ11}, \cite{Dea11}, etc.)

The concept of an isometric group action can be generalized by a \emph{singular Riemannian foliation}, 
which roughly speaking is the partition of a Riemannian manifold into smooth and equidistant submanifolds of possibly varying dimensions, called leaves (and the leaves can be thought as a generalization of the orbits of an isometric group action). 
It turns out that, while being more flexible than group actions (cf. for example \cite{Rad14}), singular Riemannian foliations
still retain a lot of the same structure of isometric group actions (cf. \cite{MR19}, \cite{GGR15}, \cite{GR15}, 
\cite{CM20}, \cite{Mor19}, etc.).

Given the action of a compact Lie group, the orbits are homogeneous spaces and thus have a very restricted topology, 
which can be employed to extrapolate topological properties of the ambient manifold (e.g. \cite{GZ12} and \cite{GYW19}). 
In \cite{GYW19}, the authors ask to what extent the leaves of a singular Riemannian foliation on a non-negatively curved space 
are also topologically restricted. In \cite{GGR15}, Galaz-Garcia and the first author proved that if $(M,\fol)$ is a closed singular Riemannian foliation on a compact, simply connected Riemannian manifold $M$, then the fundamental group of a generic leaf is a product $A\times K_2$ 
of an abelian group $A$ and a 2-step nilpotent 2-group $K_2$ - in particular, it is nilpotent. In the present paper, 
we continue exploring the topology of the leaves of singular Riemannian foliations $(M,\fol)$.

The first result states that if $M$ is simply connected, then a generic leaf $L_0$ of $\fol$ is a \emph{nilpotent space}, i.e. $\pi_1(L_0)$ acts nilpotently on $\pi_n(L_0)$ for all $n>1$:

\begin{maintheorem}\label{main-thm:leaves-nilpotent}
If $(M,\fol)$ is a closed singular Riemannian foliation on a compact, simply connected Riemannian manifold $M$,  
then the principal leaves of $\fol$ are nilpotent spaces. Furthermore, all leaves are finitely covered by a nilpotent space.
\end{maintheorem}

This answers the first part of Problem 4.8 in \cite{GYW19}:
\begin{nonumberquestion}
Let $\fol$ be a closed singular Riemannian foliation on a closed (simply connected) Riemannian manifold M of almost non-negative curvature. Are the leaves of $\fol$ finitely covered by a nilpotent space, which moreover is rationally elliptic?
\end{nonumberquestion}

Our result does not in fact use the curvature assumption. On the rationally elliptic part of the question, we make the following remarks:
\begin{enumerate}
\item The very question of whether the leaves are rationally elliptic, only makes sense the moment we know that the leaves are (virtually) nilpotent spaces: these are in fact the spaces on which rational homotopy theory applies, and the rational dichotomy of rationally elliptic vs. rationally hyperbolic spaces holds.
\item Assuming the question above to be true, and applying it to the product foliation $(M\times \mathbb{S}^n,M\times \{pts.\})$ with $M$ simply connected and almost non-negatively curved, would imply that every simply connected, almost non-negatively curved Riemannian manifold is rationally elliptic, which is the statement of the celebrated (and out of reach) Bott-Halperin-Grove Conjecture. 
In particular, the rationally elliptic part of the question is so far out of reach.
\end{enumerate}

The second result analyzes more in detail the structure of the fundamental group of a generic leaf $L_0$ 
of a singular Riemannian foliation $(M,\fol)$ with $M$ simply connected:
 
\begin{maintheorem}\label{main-thm:non-abelian part}
Let $(M,\fol)$ be a closed singular Riemannian foliation on a compact, simply connected Riemannian manifold $M$. 
If $L_0$ is a principal leaf of $\fol$, then the non-abelian part $K_2$ of the fundamental group of $L_0$ is of the form 
$$K_2\cong (\prod_{j=1}^s \Z_{2^{a_j}}\times \Z_2^b\times \prod_{i=1}^k G_i)/({\Z_2^{c}\times\Z_4^{d}}),$$
where each $G_{i}$ is isomorphic to a central product of copies of $Q_8$, with possibly one copy of $D_8$ or $\Z_4$.
\end{maintheorem}

The groups $G_i$ in the theorem are called \emph{generalized extraspecial}. These 2-groups already occur as fundamental groups of orbits of orthogonal representations and hence are impossible to avoid (e.g. $\mathrm{SO}(3)$ acting on $\mathbb{S}^4$), 
see also a family of examples from Section \ref{SS:examples}.

Finally, we extend Theorem A from \cite{GGR15} by showing that when $M$ has virtually nilpotent fundamental group, the leaves of any closed singular Riemannian foliation $(M,\fol)$ have virtually nilpotent fundamental group as well:

\begin{maintheorem}\label{main-thm:virtually nilpotent}
Suppose $(M,\fol)$ is a closed singular Riemannian foliation on compact Riemannian manifold $M$ with virtually nilpotent fundamental group. Then the leaves of $\fol$ have virtually nilpotent fundamental group as well.
\end{maintheorem}

In the fundamental paper \cite{KPT10}, the authors show that every Riemannian manifold with almost non-negative 
sectional curvature is finitely covered by a nilpotent space. With this in mind, Theorem \ref{main-thm:virtually nilpotent} gives the following steaightforward corollary:

\begin{maincorollary}
Given a closed singular Riemannian foliation $(M, \fol)$ on an almost non-negatively curved manifold $M$, the leaves have virtually nilpotent fundamental group.
\end{maincorollary}

This paper is organized as follows. In Section \ref{S:preliminaries}, we collect some preliminaries about topological results 
for singular Riemannian foliations, and the main notation for bilinear and quadratic forms we need in the proof of Theorem \ref{main-thm:non-abelian part}. In Section \ref{S:topology of leaves}, we prove Theorem \ref{main-thm:leaves-nilpotent}. 
In Section \ref{S:fundamental group}, we prove Theorem \ref{main-thm:non-abelian part} and provide a family of examples showing that the generalized extraspecial groups can indeed appear in the fundamental group of principal orbits of orthogonal representations. Finally, in Section \ref{S:nilpotent fundamental group}, we prove Theorem \ref{main-thm:virtually nilpotent}.
 
\section{Preliminaries}\label{S:preliminaries}

\subsection{Singular Riemannian foliations}

Let $M$ be a Riemannian manifold. A singular Riemannian foliation on $M$ is a partition $\fol$
of $M$ into connected, injectively immersed submanifolds called leaves such that every geodesic that starts perpendicular 
to a leaf remains perpendicular to all the leaves it meets, and moreover, M admits a family of smooth vector fields 
that spans the leaves at all points.

A singular Riemannian foliation is called closed if all of its leaves are closed in $M$. 
Given a singular Riemannian foliation $(M,\fol)$ on a complete manifold $M$ we define the \emph{dimension 
of $\fol$}, denoted $\dim\fol$, as the maximal dimension of its leaves. 
The codimension of $\fol$ is defined by $\dim M-\dim\fol$.

A leaf $L$ of the foliation $\fol$ is called regular if its dimension is maximal, or equivalently, $\dim L=\dim\fol$.
The union of all regular leaves is an open, dense and connected submanifold, which is called the principal stratum of $M$ 
and is denoted by $M_0$. The union of all other leaves is called the singular stratum of $(M,\fol)$ 
and the connected components of the singular stratum are called singular strata. 

For a closed singular Riemannian foliation $(M,\fol)$, the canonical projection $\pi:M\to M/{\fol}$
induces a metric space structure on the leaf space $M/{\fol}$, where the metric 
is given by $d_{M/{\fol}}(\pi(p),\pi(q))=d_M(L_p,L_q)$. If in addition all the leaves of $\fol$ are regular, 
then the leaf space is a Riemannian orbifold. In particular, given a closed singular Riemannian foliation $(M,\fol)$, 
the quotient space ${M_0}/{\fol}$ is an orbifold.

We then call a leaf $L\subset M_0$ \emph{principal} if it projects to a manifold point of $M_0/\fol$. Clearly, the set of principal leaves is open and dense in $M_0$.

\subsection{Slice Theorem}\label{SS:Slice Theorem}
In this section we describe the structure of a singular Riemannian foliation around a leaf. For more details, we refer the interested reader to \cite{MR19}.

Let $(M, \fol)$ be a closed singular Riemannian foliation, let $p\in M$, and let $L_p$ denote the leaf through $p$. Define the \emph{horizontal space to $\fol$ at $p$}, $\nu_pL_p\subseteq T_pM$, as the subspace perpendicular to $T_pL_p$. Then there exists a singular Riemannian foliation $(\nu_pL_p,\fol_p)$ called the \emph{infinitesimal foliation of $\fol$ at $p$}, with two important properties:
\begin{enumerate}
\item $\fol_p$ is  invariant under rescalings,
\item In an $\epsilon$-neighbourhood $\nu_p^{\epsilon}L_p$ of the origin in $\nu_pL_p$, the exponential map $\exp_p:\nu_p^\epsilon L_p\to M$ takes the leaves of $\fol_p$ onto the connected components of the intersections $L\cap \exp \nu_p^\epsilon L_p$, with $L\in \fol$. 
\end{enumerate}
Furthermore, there is a group of isometries $K\subseteq O(\nu_pL_p)$, sending leaves of $L_p$ to (possibly different) leaves of $\fol_p$, with the property that for any $v\in \nu_p^\epsilon L_p$, the leaf $L_v\in \fol_p$ satisfies the following:
\[
\exp_p(K\cdot L_v)=L_{\exp_p(v)}\cap \exp_p\nu_p^\epsilon L_p
\]
In other words, two leaves of $\fol_p$ are in the same $K$-orbit if and only if they exponentiate to different connected components of an intersection $L \cap \exp_p\nu_p^\epsilon L_p$, for some $L\in \fol$.

In \cite{MR19}, the following Slice Theorem establishes a model for a singular Riemannian foliation around a leaf:
\begin{nonumbertheorem}[Foliated Slice Theorem]
Given a closed singular Riemannian foliation $(M, \fol)$ and a point $p\in M$, let $(\nu_pL_p, \fol_p)$ be the infinitesimal foliation of $\fol$ at $p$. Then there exists a compact Lie group $K\subset O(\nu_pL_p)$ and a principal $K$-bundle $P\to L_p$ such that the foliation $\fol$ in an $\epsilon$-neighbourhood of $L_p$ is foliated diffeomorphic to
\[
(P\times_K\nu_pL,P\times_K\fol_p)
\]
\end{nonumbertheorem}

It follows directly from the Slice Theorem that all principal leaves are diffeomorphic to each other, and for any leaf $L_p$, there is a locally trivial fiber bundle $L_0\to L_p$ from a principal leaf $L_0$, whose fiber is an orbit $K\cdot L_v$ for some principal point $v\in (\nu_pL_p, \fol_p)$, and it consists of a finite disjoint union of principal leaves of $\fol_p$.

\subsection {The Molino bundle}\label{SS:molino}

Let $(M,\fol)$ be a closed singular Riemannian foliation of codimension $q$ on a compact Riemannian manifold $M$. 
The principal $\mathrm{O}(q)$-bundle $\hat{M}\to M_0$, where $\hat{M}$ is the collection of orthonormal frames 
of ${TM_0}/{T\fol}$, is called the Molino bundle. The foliation $\fol$ lifts to a foliation 
$\hat{\fol}$ on $\hat{M}$ whose leaves are diffeomorphic to the leaves of $\fol$ 
on an open dense set. Moreover, the leaves of $\hat{\fol}$ are given by fibers of a submersion 
$\theta:\hat{M}\to W$, where $W$ is the frame bundle of the orbifold ${M_0}/{\fol}$. 

Consider the fibration $\hat{\theta}:{\hat{M}}_{\mathrm{O}(q)}\to W_{\mathrm{O}(q)}$ induced by $\theta$,
where ${\hat{M}}_{\mathrm{O}(q)}={\hat{M}}\times_{\mathrm{O}(q)}\mathrm{EO}(q)$ and $W_{\mathrm{O}(q)}=W\times_{\mathrm{O}(q)}\mathrm{EO}(q)$ denote the Borel constructions of $\hat{M}$ and $W$, respectively. 
Note that $\hat{\theta}:{\hat{M}}_{\mathrm{O}(q)}\to W_{\mathrm{O}(q)}$ and $\theta:\hat{M}\to W$ 
have the same fibers and hence the fiber of $\hat{\theta}$ is diffeomorphic to $L_0$, where $L_0$ is a principal leaf 
of $\fol$. In addition, ${\hat{M}}_{\mathrm{O}(q)}$ is homotopy equivalent to ${\hat{M}}/{\mathrm{O}(q)}=M_0$ 
and $W_{\mathrm{O}(q)}$ coincides with the Haefliger's classifying space $B$ of ${M_0}/{\fol}$. 
Therefore, we get  the following fibration (up to homotopy):
$$L_0\overset{\iota_0}{\rightarrow}M_0\overset{\hat{\theta}}{\rightarrow}B.$$

\subsection{Bilinear and quadratic forms over $\Z_2$}\label{SS:quadratic}

Let $V$ be a finite dimensional vector space over a field $F$. A quadratic form on $V$ is a map $Q:V\to F$ 
such that $Q(\lambda v)=\lambda^2 Q(v)$ for all $\lambda\in F$ and $v\in V$, and moreover, the map 
$B_Q:V\times V\to F$ defined by $B_Q(u,v)=Q(u+v)-Q(u)-Q(v)$ is a bilinear form. 
Given a basis $\{v_1,\ldots,v_{\ell}\}$ of $V$, it follows that
\begin{equation}\label{eq:quadratic form}
Q(x_1v_1+\ldots+x_{\ell}v_{\ell})=\sum_{i=1}^{\ell} Q(v_i)x_i^2+\sum_{1\leq i<j\leq\ell} B_Q(v_i,v_j)x_ix_j,
\end{equation}
Two quadratic forms $Q:V\to F$ and $Q':V\to F$ are called \emph{isometric} (or equivalent) 
if there exists an invertible linear map $f:V\to V$ such that $Q(v)=Q'(f(v))$ for all $v\in V$. 

Finally, given quadratic forms $Q:V\to F$ and $Q':V'\to F$, one defines the \emph{orthogonal sum} $Q\oplus Q':V\oplus V'\to F$ 
by $(Q\oplus Q')(v,v'):=Q(v)+Q'(v')$.

\section{The topology of leaves}\label{S:topology of leaves}

Let $(M,\fol)$ be a closed singular Riemannian foliation on a compact, simply connected Riemannian manifold $M$. 
The goal is to prove Theorem \ref{main-thm:leaves-nilpotent}, that the principal leaves are nilpotent manifolds.

We begin by collecting some of the results proved in \cite{GGR15} about the fundamental group of the principal leaves of $\fol$.

\subsection{Known results on the topology of leaves}\label{SS:known-results}

Since the fundamental group of $M$ does not change if we delete the strata of $\fol$ with codimension $>2$,
we can assume that we only have singular strata of codimension $\leq 2$. Furthermore, it is known that there are no strata 
of codimension one, which reduces $\fol$ to only having strata of codimension two. 

Let $\Sigma_1,\ldots,\Sigma_m$ denote the singular strata of $\fol$ of codimension two.
For $i=1,\ldots, m$, choose a singular leaf $L'_i$ in $\Sigma_i$, and let $L_i$ be a principal leaf at some distance $\epsilon_i$ 
from $L'_i$. For $\epsilon_i$ small enough, the foot-point projection $\pi_i:L_i\to L'_i$ is a circle bundle.
Fix a point $p_i\in L_i$, and let $[c_i]\in \pi_1(L_i,p_i)$ be the element represented by the fiber $c_i$ 
of $\pi_i$ through $p_i$.

Fixing a principal leaf $L_0$ and $p_0\in L_0$, we can choose, for each $i=1,\ldots, m$,  a diffeomorphism $h_i:L_i\to L_0$, 
and define $k_i=(h_i)_*([c_i])\in\pi_1(L_0,p_0)$. The group $K$ generated by the elements $k_i$ is then a normal subgroup 
of $\pi_1(L_0,p_0)$. Furthermore, there exists a homotopy fibration
\[
L_0\overset{\iota_0}{\rightarrow}M_0\overset{\hat{\theta}}{\rightarrow}B,
\]
as described in Section \ref{SS:molino}. One has the following (see the proof of Theorem A in \cite{GGR15}):
\begin{enumerate}
\item $\pi_1(L_0,p_0)$ is generated by the subgroup $K$ and the image of the boundary map 
$\partial:\pi_2(B,b_0)\to\pi_1(L_0,p_0)$.
\item $H:=\im(\partial)$ is central in $\pi_1(L_0,p_0)$
\item Any two non-commuting generators $k_i$ and $k_j$ of $K$ satisfy $k_ik_j=k_j^{-1}k_i$.
\item Let $N\subseteq K$ be the subgroup generated by the non-central $k_i$'s, and let $Z_{(2)}$ 
denote the Sylow $2$-subgroup of $Z(K)$. Then $\pi_1(L_0,p_0)$ is nilpotent, and equal to $A\times K_2$, 
where $A$ is abelian and $K_2=N\cdot Z_{(2)}$.
\end{enumerate}

\subsection{Proof of Theorem \ref{main-thm:leaves-nilpotent}}

As discussed in Section \ref{SS:known-results}, the principal leaves of $\fol$ have nilpotent fundamental groups. 
As a first step towards the proof of Theorem \ref{main-thm:leaves-nilpotent}, we prove that the principal leaves 
are nilpotent spaces:

\begin{proposition}\label{P:princ-leaves-nilp}
Suppose $(M,\fol)$ is a closed singular Riemannian foliation on a compact, simply connected Riemannian manifold $M$. 
Let $L_0$ denote a principal leaf of $\fol$ and let $p_0\in L_0$. Then $\pi_1(L_0,p_0)$ acts trivially on $\pi_n(L_0,p_0)$ 
for $n\geq 2$.
\end{proposition}

\begin{proof}
Let $[\gamma]\in\pi_1(L_0,p_0)$ and $[\omega]\in\pi_n(L_0,p_0)$. The goal is to prove that $[\gamma]$
acts trivially on $[\omega]$. By the discussion in Section \ref{SS:known-results}, we may assume that either $[\gamma]\in H$ 
or $[\gamma]=k_i$ for some $i$. 
\par     
First, consider the case in which $[\gamma]=k_i$ for some $i$. Note that ${\bf p}_i:=\pi_i\circ h_i^{-1}:L_0\to L'_i$ 
is a circle bundle whose fiber is represented by $k_i$. This means that $k_i\in\ker(({\bf p}_i)_*)$, where $({\bf p}_i)_*$
is the induced map on $\pi_n$. Hence we have:
$$({\bf p}_i)_*([\gamma]\cdot[\omega])=({\bf p}_i)_*(k_i\cdot[\omega])=(({\bf p}_i)_*(k_i))\cdot(({\bf p}_i)_*([\omega]))=({\bf p}_i)_*([\omega]).$$
By the long exact sequence of homotopy groups associated to the fibration 
$\s^1\to L_0\overset{{\bf p}_i}{\rightarrow} L'_i$, it follows that the homomorphism 
$({\bf p}_i)_*$ is injective in $\pi_n$ for $n\geq 2$. 
This, together with $({\bf p}_i)_*([\gamma]\cdot[\omega])=({\bf p}_i)_*([\omega])$, implies that $[\gamma]$ 
acts trivially on $[\omega]$.
\par
Suppose now that $[\gamma]\in H=\im(\partial)$ and choose $[\beta]\in\pi_2(B,b_0)$ such that $[\gamma]=\partial([\beta])$.
Consider the fibration 
$$L_0\overset{\iota_0}{\rightarrow}M_0\overset{\hat{\theta}}{\rightarrow}B.$$
Note that the action of $\pi_1(L_0,p_0)$ on $\pi_n(L_0,p_0)$ satisfies $[\gamma]\cdot[\omega]=(\iota_0)_*([\gamma])\cdot[\omega]$ (see \cite[Exercise 4.3.10]{Hat02}). Therefore, 
$$[\gamma]\cdot[\omega]=(\iota_0)_*([\gamma])\cdot[\omega]=(\iota_0)_*(\partial([\beta]))\cdot[\omega]=e\cdot[\omega]=[\omega].$$
This completes the proof.
\end{proof}

Moving to the non-principal leaves, we first prove that every leaf has a virtually nilpotent fundamental group.

\begin{lemma}\label{L:other-leaves}
Suppose $(M,\fol)$ is a closed singular Riemannian foliation with principal leaf $L_0$. If $\pi_1(L_0)$ is virtually nilpotent, 
then so is the fundamental group $\pi_1(L)$ of every leaf $L$ of $\fol$.
\end{lemma}

\begin{proof}
For any leaf $L$ of $\fol$, the foliated Slice Theorem (cf. Section \ref{SS:Slice Theorem}) implies that there is a fibration $L_0\to L$ whose fiber $F$ 
has finitely many connected components. From the long exact sequence in homotopy one then has
\[
\pi_1(L_0)\to \pi_1(L)\to \pi_0(F)
\]
from which it follows that $\pi_1(L)$ is a finite extension of a quotient of $\pi_1(L_0)$, therefore it is virtually nilpotent as well.
\end{proof}

\begin{proof}[Proof of Theorem \ref{main-thm:leaves-nilpotent}]
The statement about principal leaves has been proved in Proposition \ref{P:princ-leaves-nilp}, so we now have to only consider non-principal leaves.

Given a leaf $L$, choose $p\in L$. Recall that, by the Foliated Slice Theorem (cf. \ref{SS:Slice Theorem}), there is a locally trivial fibration $\phi:L_0\to L$ whose fiber $F$ has finitely many connected components, all diffeomorphic to a principal leaf of the infinitesimal foliation $(\nu_p L_p,\fol_p)$. Furthermore, the action $\pi_1(L)\to \operatorname{Diff}(F)$ induces an action $\pi_1(L)\to \operatorname{Aut}(\pi_*(F))$, which factors as $\pi_1(L)\stackrel{\psi}{\to} \pi_0(K)\to \operatorname{Aut}(\pi_*(F))$. In particular:
\begin{enumerate}
\item The subgroup $G_1:=\ker\psi \subseteq \pi_1(L)$ has finite index in $\pi_1(L)$ and it acts trivially on $\pi_*(F)$.
\item The fibration induces a map $\pi_1(L_0)\stackrel{\phi_*}{\to} \pi_1(L)\to \pi_0(F)$. Thus $G_2:=\phi_*(\pi_1(L_0))$ is a nilpotent subgroup of $\pi_1(L)$ with finite index.
\end{enumerate}
Consider $G:=G_1\cap G_2\subseteq \pi_1(L)$, which is by the points above a nilpotent subgroup with finite index. We will now show that $G$ acts nilpotently on each $\pi_n(L)$, i.e. the \emph{lower central series} $\Gamma^m_G(\pi_n(L))\subseteq\pi_n(L)$ defined iteratively by
\[
\Gamma_G^1(\pi_n(L))=\pi_n(L), \qquad \Gamma_G^{m+1}(\pi_n(L))=\{\gamma\cdot \alpha-\alpha\mid \gamma\in G,\,\alpha\in  \Gamma_G^{m}(\pi_n(L))\}
\]
eventually becomes trivial.

Consider the long exact sequence
\[
\cdots\to\pi_n(F)\to\pi_n(L_0)\stackrel{\phi_*}{\to}\pi_n(L)\stackrel{\partial}{\to}\pi_{n-1}(F)\to\cdots
\]
Let $\alpha\in \pi_n(L)$, and $\gamma=\phi_*(\gamma_0)\in G$, where $\gamma_0\in \pi_1(L_0)$. Recall that $\partial(\gamma\cdot \alpha)=\gamma\cdot \partial(\alpha)$, where the action on the left is $\pi_1(L)$ acting on $\pi_*(L)$, while on the right we have the $\pi_1(L)$-action on $\pi_*(F)$. Since $G\subseteq G_1$, we have
\[
\partial(\gamma\cdot\alpha)=\partial (\alpha)\quad\Rightarrow \quad \partial(\gamma\cdot \alpha-\alpha)=0
\]
and therefore
\[
\Gamma_G^2(\pi_n(L))\subseteq \ker (\partial)=\phi_*(\pi_n(L_0))=\phi_*(\Gamma_{\pi_1(L_0)}^1(\pi_n(L_0))).
\]
Finally, we notice that if $\alpha=\phi_*(\alpha_0)$ with $\alpha_0\in \pi_n(L_0)$ then
\[
\gamma\cdot \alpha=(\phi_*(\gamma_0))\cdot(\phi_*(\alpha_0))=\phi_*(\alpha_0)\Rightarrow \gamma\cdot \alpha-\alpha=\phi_*(\gamma_0\cdot \alpha_0-\alpha_0).
\]

By induction on $m$, one then has
\[
\Gamma_G^{m+1}(\pi_n(L))\subseteq \phi_*\big(\Gamma_{\pi_1(L_0)}^{m}(\pi_n(L_0))\big).
\]
Since by Proposition \ref{P:princ-leaves-nilp}, $\Gamma_{\pi_1(L_0)}^{2}(\pi_n(L_0))=0$, we have $\Gamma_G^{3}(\pi_n(L))=0$ which proves that $G$ acts nilpotently on $\pi_n(L)$, hence finishing the proof.
\end{proof}

\section{Fundamental groups of the principal leaves}\label{S:fundamental group}

This section consists of two parts. The first part is devoted to the proof of Theorem \ref{main-thm:non-abelian part}.
In the second part, we provide examples of singular Riemannian foliations whose principal leaves have fundamental groups 
of the form discussed in Theorem \ref{main-thm:non-abelian part}.

Suppose that $(M,\fol)$ is a closed singular Riemannian foliation on a compact, simply connected Riemannian manifold $M$. 
Fix a principal leaf $L_0$ of $\fol$ and $p_0\in L_0$. Let $N$ and $K_2$ be the subgroups of $\pi_1(L_0,p_0)$ 
discussed in Section \ref{SS:known-results}.

Consider the graph $\Gamma$ with vertices the generators of $N$ and an edge between $k_i$ and $k_j$ 
if and only if $k_ik_jk_i^{-1}=k_j^{-1}$. Note that for every generator $k_i$ of $N$, there exists another generator 
which does not commute with $k_i$. Therefore, $\Gamma$ does not contain any isolated vertices. 
Note moreover that for every connected component $\Gamma_i$ of $\Gamma$, all vertices of $\Gamma_i$ 
square to the same element $c_i$. In addition, by proof of Theorem A in \cite{GGR15}, for any generator $k_i$ of $N$,
we have $k_i^4=1$ and $k_i^2$ is central in $K$. Therefore, $c_i$ is a central element of $N$ of order two.
Altogether, we get that there is a map $C:\pi_0(\Gamma)\to Z(N)$ defined by $C(\Gamma_i)=c_i$.

\begin{notation}\label{N:N_c}
From now on, we fix an element $c$ of $Z(N)$ which is of the form $k_i^2$ for some generator $k_i$ of $N$.
Moreover, $N_c$ denotes the subgroup of $N$ that is generated by all the vertices in $\Gamma_c:=C^{-1}(c)$.
\end{notation}

Recall that given a group $G$, the \emph{Frattini subgroup} $\Phi(G)$ is the intersection of all the maximal subroups of $G$. Furthermore, we recall the following:

\begin{definition}\label{def:Generalized ES}
A $2$-group $G$ is called \emph{generalized extraspecial} if $\Phi(G)$ is central, and $\Phi(G)=[G,G]=\Z_2$.
\end{definition}

We prove two important properties of the groups $N_c$. 

\begin{lemma}\label{L:abt-N_c}
Let $\{N_c\}_{c\in \textrm{Im}(C)}$ be the collection of groups defined above. Then:
\begin{enumerate}
\item For $c\neq c'$, the groups $N_{c}$ and $N_{c'}$ commute.
\item Each $N_c$ is a generalized extraspecial $2$-group.
\end{enumerate}
\end{lemma}

\begin{proof}
First, we prove Statement (1). Let $k_1,\ldots, k_{\ell}$ be the generators of $N_c$, and $k'_1,\ldots, k'_r$ be the generators
of $N_{c'}$. As vertices of $\Gamma$, there is no edge between any $k_i$ and any $k'_j$, which means that each $k_i$ commutes with any $k'_j$ in $K$. Hence the result follows.
\par
As for Statement (2), if $k_1,\ldots, k_{\ell}$ denote the generators of $N_c$, then $V={N_c}/{\langle c\rangle}$ 
is isomorphic to $\Z_2^{\ell}$ and is generated by $[k_1],\ldots,[k_{\ell}]$. 
It follows that $N_c$ fits into a short exact sequence
\[
1
\to \langle c\rangle \to N_c\to V\to 1
\]
and in particular one has that both $N_c^2:=\langle g^2\mid g\in N_c\rangle$ and the commutator subgroup $[N_c,N_c]$ 
coincide with $\langle c \rangle\simeq \Z_2$. Therefore, the same is true for the Frattini subgroup $\Phi(N_c)$
since for a $2$-group $G$, one has $\Phi(G)=G^2\cdot [G,G]$.
\end{proof}

Given generalized extraspecial groups $G_1$ and $G_2$, with Frattini subgroups generated by $c_1$ and $c_2$, respectively, define the \emph{central product} $G_1*G_2$ by $G_1*G_2:={(G_1\times G_2)}/\langle (c_1,c_2)\rangle$. 
This is again a generalized extraspecial group, since
\[\Phi(G_1*G_2)=\Phi(G_1)\times_{\Z_2}\Phi(G_2) \cong \Z_2.\]

The $*$ operation is furthermore associative, and thus it makes sense to define, for a generalized extraspecial group $G$, 
the central product powers
\[
(G)^{*m}:=\underbrace{G*G*\ldots*G}_\text{m\textrm{ times}}
\]

Generalized extraspecial 2-groups are, as the name suggests, a generalization of \emph{extraspecial $2$-groups}, 
that is 2-groups such that $\Phi(G)=Z(G)=[G,G]\cong \Z_2$. These groups have been thoroughly studied at least since the 60's \cite{Hal56}. They are extremely simple: an extraspecial group has the form $(Q_8)^{*m}$ or $(Q_8)^{*(m-1)}*D_8$ 
for some $m\geq 1$, where $Q_8$ is the quaternion group and $D_8$ is the dihedral group of order $8$
(cf. Theorem 2.2.11 of \cite{LGM05}). It then follows from Lemma 3.2 in \cite{Sta02} that

\begin{theorem}
A generalized extraspecial $2$-group is of the form $G\times \Z_2^n$, where $G$ is one of
\[
Q_8^{*m},\qquad Q_8^{*(m-1)}*D_8,\qquad Q_8^{*(m-1)}*\Z_4.
\]
\end{theorem}

\subsection{The associated quadratic form}\label{SS:associated quadratic form}

Let $G$ be a generalized extraspecial $2$-group with $\Phi(G)=G^2=\langle c\rangle$,
and let $V:=G/{\langle c\rangle}$. It is easy to check that $V$ is a vector space over $\Z_2$. 

Define the function $Q_G:V\to\Z_2$ by $Q_G([g])=k$, where $g^2=c^k$.  Since $c$ is central in $G$ and has order two, 
for any $g\in G$, we have $(cg)^2=cgcg=c^2g^2=g^2$ and thus $Q_G([cg])=Q_G([g])$. 
Therefore, $Q:=Q_G$ is well-defined and in fact a quadratic form as defined in Section \ref{SS:quadratic}. 
Furthermore, the bilinear form $B_Q$ associated to $Q$ (cf. Section  \ref{SS:quadratic}) satisfies 
$$ghg^{-1}h^{-1}=c^{B_Q([g],[h])},~{\text{for}}~g,h\in G.$$ 
In order to see this, note that both $g^2$ and $h^2$ are central elements of $G$. Therefore,
$$c^{B_Q([g],[h])}=c^{Q([g]+[h])}c^{-Q([g])}c^{-Q([h])}=(gh)^2g^{-2}h^{-2}=ghg^{-1}h^{-1}.$$

%Conversely, every non-trivial quadratic form $Q:V\cong\Z_2^{\ell}\to\Z_2$ produces a generalized extraspecial group $G_Q$ with $G_Q/\Phi(G_Q)\simeq V$, as follows. 
%Choosing a basis $\{v_1,\ldots, v_{\ell}\}$ for $V$, let $G_Q$ be the group with the presentation
%%\todo{I just called it $G_Q$ right away}
%\begin{equation}
%\label{E:presentation}G_Q=\langle c, g_1,\ldots, g_{\ell}\mid g_i^2=c^{Q(v_i)}, c^2=1, cg_i=g_ic, g_ig_j=c^{B_{ij}}g_jg_i\rangle.
%\end{equation}
%where $B_{ij}=B_G(v_i,v_j)$. Since $Q$ is non-trivial, the subgroup $\Phi(G_Q)$ is isomorphic to $\langle c\rangle\cong\Z_2$ 
%which is central in $G_Q$. Hence $G_Q$ is generalized extraspecial.
%
%\todo{Moved to appendix}
%It is an easy, and probably standard fact, that:
%\begin{enumerate}
%\item The isomorphism class of $G_Q$ does not depend on the choice of the basis (cf. Lemma \ref{L:GQ-well-defined}).
%\item Given a vector space $V$ over $\Z_2$, the correspondence $Q\mapsto G_Q$ induces a bijection between non-trivial quadratic forms on $V$.
%and generalized extraspecial groups extending $V$ (cf. Proposition \ref{prop:1-1}).
%\item Given a quadratic form  $Q:V\to\Z_2$ splitting as a sum $Q=Q_1\oplus Q_2$, with  $Q_1$ and $Q_2$ are non-trivial, then $G_Q\simeq G_{Q_1}*G_{Q_2}$ (cf. Proposition \ref{prop:orthogonal sum group}).
%\end{enumerate}
%

The quadratic form of each generalized extraspecial group can be explicitly computed. For this, consider the quadratic forms:
\begin{alignat*}{3}
 & H_+: \Z_2^2\to \Z_2 && \qquad H_-:\Z_2^2\to\Z_2 && \qquad \q:\Z_2\to \Z_2\\
 & H_+(x,y)=xy && \qquad H_-(x,y)=x^2+y^2+xy && \qquad \q(x)=x^2.
\end{alignat*}
We have the following:

\begin{proposition}\label{prop:quad-forms}
Suppose that $G$ is a generalized extraspecial $2$-group and let $V:= G/\Phi(G)$. 
\begin{enumerate}
\item If $G=(Q_8)^{*m}$, then $V\simeq\Z_2^{2m}$ and
\[
Q_G=H_-^{\oplus m}=\begin{cases}
H_+^{\oplus m}& m~{\rm{even}}\\
H_-\oplus H_+^{\oplus (m-1)} & m~{\rm{odd}}
\end{cases}
\]
\item If $G=(Q_8)^{*(m-1)}*D_8$, then $V\simeq\Z_2^{2m}$ and
\[
Q_G=H_-^{\oplus (m-1)}\oplus H_+=\begin{cases}
H_+^{\oplus m} & m~{\rm{odd}}\\
H_-\oplus H_+^{\oplus (m-1)}& m~{\rm{even}}
\end{cases}
\]
\item If $G=(Q_8)^{m}*\Z_4$, then $V\simeq\Z_2^{2m+1}$ and
\[
Q_G=H_+^{\oplus m}\oplus \q=H_-^{\oplus m}\oplus \q\]
\item If $G=G'\times \Z_2^n$ with $G'$ as in the previous points, then $V\simeq V'\oplus \Z_2^n$ and $Q_G=Q_{G'}\oplus0^{\oplus n}$.
\end{enumerate}
\end{proposition}

\begin{proof}
This proposition follows easily from the following straightforward facts:
\begin{enumerate}
\item For $G=Q_8$, $G/\Phi(G)\simeq \Z_2^2$ and $Q_G=H_-$.
\item For $G=D_8$, $G/\Phi(G)\simeq \Z_2^2$ and $Q_G=H_+$.
\item For $G=\Z_4$, $G/\Phi(G)\simeq \Z_2$ and $Q_G=\q$.
\item Given $G_1$ and $G_2$ with quotients $V_i=G_i/\Phi(G_i)$, one has
\[
(G_1*G_2)/\Phi(G_1*G_2)=V_1\oplus V_2\quad\textrm{and}\quad Q_{G_1*G_2}=Q_{G_1}\oplus Q_{G_2}.
\]
\item Given $G$ with quotient $V=G/\Phi(G)$, one has
\[
(G\times \Z_2^n)/\Phi(G\times \Z_2^n)\simeq V\oplus \Z_2^n\quad\textrm{and}\quad Q_{G\times \Z_2^n}=Q_G\oplus 0^{\oplus n}.
\qedhere\]
\end{enumerate}
\end{proof}

\begin{remark}\label{remark:N_c}
The group $N_c$ discussed above is generated by elements of order four, that is the $k_i$'s. 
Moreover, for each $k_i$, there exists $k_j$ such that $k_ik_jk_i^{-1}k_j^{-1}=c$.
This is reflected in the corresponding quadratic form $Q:V\to\Z_2=\{0,1\}$ as follows.
There exists a basis $\{v_1,\ldots,v_{\ell}\}$ of $V\cong\Z_2^\ell$ with the property that $Q(v_i)=1$ for all $i$, 
and for each $v_i$, there exists $v_j$ such that $B_Q(v_i,v_j)=1$. We call such quadratic forms \emph{admissible}. 
\end{remark}

The next step consists of understanding which of the quadratic forms in Proposition \ref{prop:quad-forms} is admissible. 
We start by reducing the problem to quadratic forms without trivial summands:
\begin{lemma}\label{lem:splitting}
Let $Q:V\to\Z_2$ be a quadratic form. If there exists a splitting $V=V_1\oplus V_2$ such that $Q$ splits as $Q=q\oplus 0^{\oplus n}$ with $Q|_{V_1}=q$ and $Q|_{V_2}=0^{\oplus n}$, then $Q$ is admissible if and only if $q$ is admissible.
\end{lemma}

\begin{proof}
Suppose that $Q$ is admissible and choose a basis
$$\{(v_1,w_1),\ldots, (v_{m+n},w_{m+n})\}$$
of $V_1\oplus V_2$ with the property that $Q(v_i,w_i)=1$, and for every $(v_i,w_i)$ there exists $(v_j,w_j)$
with $B_Q((v_i,w_i),(v_j,w_j))=1$. After possibly rearranging basis elements of $V_1\oplus V_2$, 
we may assume that $\{v_1,\ldots, v_m\}$ forms a basis for $V_1$. 
Since $Q(v_i,w_i)=q(v_i)$ and $B_Q((v_i,w_i),(v_j,w_j))=B_q(v_i,v_j)$, the basis $\{v_1,\ldots, v_m\}$ of $V_1$ 
is admissible for $q$. On the other hand, if $\{v_1,\ldots, v_m\}$ is admissible for $q$ and $\{w_1,\ldots, w_n\}$ 
is any basis of $V_2$, then 
$$\{(v_i,{\bf{0}})\mid  i=1,\ldots, m\}\cup\{(v_1,w_j)\mid j=1,\ldots, n\},$$
forms an admissible basis for $Q$.\end{proof}

We now apply Lemma \ref{lem:splitting} to classify the admissible quadratic forms.

\begin{theorem}\label{thm:admissible}
Any admissible quadratic form $Q:\Z_2^{\ell}\to\Z_2$ is isometric to one of the following, up to orthogonal sum 
with $0^{\oplus n}$:
\begin{equation}\label{eq:admissible}
H_+^{\oplus m}~(m\geq 2),\qquad H_-\oplus H_+^{\oplus m-1},\qquad H_+^{\oplus m}\oplus \q~(m\geq 2).
\end{equation}
\end{theorem}

\begin{proof}
Since the quadratic forms over $\Z_2$ are classified (see Proposition \ref{prop:quadratic form}), 
we only need to check the admissibility condition. By Lemma \ref{lem:splitting}, we may assume that $Q$ does not split 
as $q\oplus 0^{\oplus m}$. We break the proof into cases.

\smallskip

{\textbf{Case 1:}} $Q=H_-\oplus H_+^{\oplus m-1}$, where $2m=\ell$. The quadratic form $Q$ is given by
$$Q(x,y,z_1,z_2,\ldots,z_{2m-2})=x^2+xy+y^2+z_1z_2+\ldots+z_{2m-3}z_{2m-2}.$$
Let $e_1,\ldots,e_{\ell}$ denote the standard basis elements of $\Z_2^{\ell}$ and consider the following basis:
\begin{align*}
 & v_1=e_1+e_2, \quad v_2=e_3+e_4,\quad\ldots\quad  v_{m}=e_{2m-1}+e_{2m},\\
 & v_{m+1}=e_1,\\
 & v_{m+2}=e_1+e_3, \quad v_{m+3}=e_1+e_5,\quad \ldots\quad v_{2m}=e_1+e_{2m-1}.
\end{align*}
Then $Q(v_i)=1$ for all $i$, and for every $v_i$, there exists $v_j$ such that $B_Q(v_i,v_j)=1$. Hence $Q$ is admissible. 

\smallskip

{\textbf{Case 2:}} $Q=H_+^{\oplus m}$, where $2m=\ell$. Note that the only element of $\Z_2^2$ 
that is mapped to $1$ by $H_+$ is $(1,1)$. Therefore, $H_+$ is not admissible. However, if $m\geq 2$, 
then the following basis of $\Z_2^{\ell}$ is admissible for $Q$:
\begin{align*}
  &v_1=e_1+e_2, \quad v_2=e_3+e_4\quad \ldots\quad v_{m}=e_{2m-1}+e_{2m},\\
 &v_{m+1}=e_1+e_{2m-1}+e_{2m},\\
 & v_{m+2}=e_1+e_2+e_4, \quad v_{m+3}=e_3+e_4+e_6,\ldots \\
 &v_{2m}=e_{2m-3}+e_{2m-2}+e_{2m}.
\end{align*}

\smallskip

{\textbf{Case 3:}} $Q=H_+^{\oplus m}\oplus \q$, where $m\geq 2$ and $2m+1=\ell$. 
Let $\{v_1,\ldots, v_{2m}\}$ denote the basis constructed for $H_+^{\oplus m}$ in Case 2, 
and let $v_{2m+1}=e_1+e_{2m+1}$. Then $\{v_1,\ldots, v_{2m+1}\}$ forms an admissible basis for $Q$.

For $Q=H_+\oplus \q$, the elements with non-zero quadratic form are $(1,1,0)$, $(1,0,1)$, $(0,1,1)$, $(0,0,1)$. 
Among these, the only vectors with non-zero bilinear form are the first three, which are linearly dependent 
and thus do not form a basis. Hence $H_+\oplus \q$ is not admissible.
\end{proof}

Recall that the group $N_c$ (cf. Notation \ref{N:N_c}) is a generalized extraspecial group with an admissible basis. 
From the previous theorem, we then get:

\begin{corollary}\label{C:N_c}
If $N_c$ is a generalized extraspecial group whose corresponding quadratic form is admissible,
then, up to a direct product with copies of $\Z_2$, the group $N_c$ is isomorphic to one of the following:
\begin{equation}\label{eq:N_c}
(Q_{8})^{*m_1},\qquad (Q_{8})^{*(m_1-1)}*D_8~(m_1\geq 2),\qquad (Q_{8})^{*m_1}*\Z_4~(m_1\geq 2).
\end{equation}
\end{corollary}

\begin{proof}
This follows trivially by comparing the quadratic forms in Proposition \ref{prop:quad-forms} with the classification of admissible quadratic forms in Theorem \ref{thm:admissible}.
\end{proof}

Finally, we prove Theorem \ref{main-thm:non-abelian part}. 

%\begin{theorem}
%A generalized extraspecial group $G$ is isomorphic to $G'\times \Z_2^{m_2}$, where $G'$ is one of the following.
%\begin{equation}
%(Q_{8})^{*m_1},\hspace{0.5cm}(Q_{8})^{*(m_1-1)}*D_8,\hspace{0.5cm}(Q_{8})^{*m_1}*\Z_4.
%\end{equation}
%\end{theorem}

%\begin{proof}
%By Lemma \ref{prop:1-1}, $G$ is isomorphic to $G_Q$ where $Q$ is one of the quadratic forms in Proposition \ref{prop:quadratic form}. By Lemma \ref{lem:splitting} it follows that $G=G_{Q'}\times \Z_2^{m_2}$ where $Q'$ is one of
%$$H_+^{\oplus m_1},\hspace{0.5cm}H_-\oplus H_+^{\oplus m_1-1}, \hspace{0.5cm}H_+^{\oplus m_1}\oplus Q_1.$$
%
%By Proposition \ref{prop:orthogonal sum group} we have that $G_{Q'}$ is isomorphic to one of:
%\[
%(G_{H_+})^{*m_1},\qquad G_{H_-}*(G_{H_+})^{*(m_1-1)},\qquad (G_{H_+})^{*m_1}*G_{Q_1}
%\]
%Thus it is enough to compute $G_{H_-}$, $G_{H_+}$, $G_{Q_1}$:
%\begin{align*}
%G_{H_-}=&\langle c, g_1, g_2\mid g_i^2=c, c^2=1, cg_i=g_ic, g_1g_2=cg_2g_1\rangle\cong Q_8\\
%G_{H_+}=&\langle c', h_1, h_2\mid h_i^2=c', (c')^2=1, c'h_i=h_ic', h_1h_2=c'h_2h_1\rangle\cong D_8\\
%G_{Q_1}=&\langle c'', h\mid h^2=c'', (c'')^2=1, c''h=hc''\rangle\cong \Z_4.
%\end{align*}
%Thus $G_{Q'}$ is isomorphic to one of
%\begin{align*}
%(D_8)^{*m_1}&\cong
%\begin{cases}
%(Q_{8})^{*m_1} & m_1~{\text{even}}\\
%(Q_{8})^{*(m_1-1)}*D_8  & m_1~{\text{odd}}
%\end{cases}\\
%Q_8*(D_{8})^{*(m_1-1)}&\cong
%\begin{cases}
%(Q_{8})^{*m_1} & m_1~{\text{odd}}\\
%(Q_{8})^{*(m_1-1)}*D_8  & m_1~{\text{even}}
%\end{cases}\\
%\Z_4* (D_{8})^{*m_1}&\cong \Z_4* (Q_{8})^{*m_1}
%\end{align*}
%\end{proof}

\begin{proof}[Proof of Theorem \ref{main-thm:non-abelian part}]
Fix $p_0\in L_0$. As discussed in Section \ref{SS:known-results}, the non-abelian part $K_2$ of $\pi_1(L_0,p_0)$ 
is a $2$-group of the form $K_2=N\cdot Z_{(2)}$, where $N$ is generated by the non-central generators of $K$
and $Z_{(2)}$ denotes the Sylow $2$-subgroup of $Z(K)$. Furthermore, by the discussion in Section \ref{S:fundamental group}, $N=N_{c_1}\cdot \ldots\cdot N_{c_k}$, where the elements $c_i\in Z(K)$ have order two.
By Corollary \ref{C:N_c}, each $N_{c_i}$ is of the form $G_i\times \Z_2^{a_i}$, where $G_i$ is one of the groups 
listed in Equation \eqref{eq:N_c}. Let $a=\sum_i a_i$. Finally, since all the groups $N_{c_i}$ commute 
with one another by Lemma \ref{L:abt-N_c}, one has $N_{c_i}\cap N_{c_j}\subseteq Z(N_{c_i})\cap Z(N_{c_j})$, 
and $Z(N_{c_i})\subseteq Z(K_2)$. Therefore
\[
K_2\cong (Z_{(2)}\times \prod_{i=1}^kN_{c_i})/Z'=(Z_{(2)}\times \Z_2^a\times \prod G_i)/Z',
\]
where $Z'\subseteq Z_{(2)}\times \prod_i Z(N_{c_i})$ is the subgroup of $K_2$ generated by the intersections $H_{ij}=N_{c_i}\cap N_{c_j}$ and $H_{0j}=Z_{(2)}\cap N_{c_j}$. Since the groups $H_{ij}$, $H_{0j}$ are all abelian 
and central, commute with one another, and have elements of order $2$ or $4$ (because $Z(N_{c_i})=\Z_2^{a_i}\times \Z_2$ 
or $\Z_2^{a_i}\times \Z_4$) it follows that $Z'=\Z_2^\alpha\times Z_4^\beta$ for some $\alpha$ and $\beta$.
\end{proof}

%We need to compute $K_2=N\cdot Z_{(2)}=N_{c_1}$. Note that $N$ and $Z_{(2)}$ commute, and the intersection $N\cap Z_{(2)}$
%is the subgroup $Z(N)$. Therefore,
%$$K_2=N\cdot Z_{(2)}\cong (N\times Z_{(2)})/{Z(N)}.$$
%By Proposition \ref{prop:N}, there is an isomorphism $\phi: (G_{1}\times\ldots\times G_{r})\times\Z_2^{\eta}\to N$ 
%for some $G_{i}$ as in Display \eqref{eq:N_c} and some $\eta\geq 0$. 
%Moreover, $Z(N)$ is isomorphic to $\Z_2^{\alpha}\times\Z_4^{r-\alpha}\times \Z_2^\eta$,
%where $\Z_2^{\alpha}\subseteq\langle c_1\rangle\times\ldots\times\langle c_r\rangle$. Hence we get the isomorphism
%$$K_2\cong{\left(\left(\prod_{i=1}^r G_{i}\right)\times\Z_2^{\eta}\times Z_{(2)}\right)}{\Big{/}}{(\Z_2^{\alpha}\times\Z_4^{r-\alpha}\times\Z_2^\eta)}.$$
%
%Notice that $\phi(\Z_2^\eta)\subseteq N\cap Z_{(2)}$, and thus it makes sense to define
%
%$$\rho:K_2\longrightarrow{\left(\left(\prod_{i=1}^r G_{i}\right)\times Z_{(2)}\right)}{\Big{/}}{(\Z_2^{\alpha}\times\Z_4^{r-\alpha})}$$
%by $\rho([g,h,z])=[g,\phi(h)z]$, where $g\in\prod_{i=1}^r G_{i}$, $h\in\Z_2^{\eta}$ and $z\in Z_{(2)}$.
%
%It can be easily checked that $\rho$ is an isomorphism. Since $Z_{(2)}\cong\prod_{j=1}^s\Z_2^{\beta_j}$,
%we conclude that $K_2$ is isomorphic to the central product of $\prod_{i=1}^r G_{i}$ and $\prod_{j=1}^s\Z_2^{\beta_j}$
%with respect to $\Z_2^{\alpha}\times\Z_4^{r-\alpha}$, as required.

\subsection{Examples of fundamental groups of principal leaves}\label{SS:examples}

The family of examples below shows that the non-abelian groups $G_i$ discussed in Theorem \ref{main-thm:non-abelian part} 
actually arise as fundamental groups of principal leaves of homogeneous singular Riemannain foliations.

Let $\{e_1,\ldots, e_n\}$ be the standard basis of $\R^n$. The Clifford algebra $Cl(0,n)$ on $\R^n$
is defined as the associative algebra generated by $e_1,\ldots, e_n$, where multiplication of the elements $e_i$ 
is given by:
$$e_i^2=-1,\hspace{0.3cm}e_ie_j=-e_je_i.$$
Consider the subset $E(n)=\{\pm e_{i_1}\ldots e_{i_{2k}}\}\subseteq Cl(0,n)$ containing products of even numbers of the $e_i$'s. This is easily seen to be a group under the product of $Cl(0,n)$. In \cite{CHM09}, Czarnecki, Howe, and McTavish prove that for the action of $G=\mathrm{SO}(n)\times\mathrm{SO}(n)$ 
on $M_{n\times n}(\R)$ defined by $(g,h)\cdot A=g^TAh$, the fundamental group of a principal orbit is of the form $E(n)\times\Z_2$. In this section, we investigate the structure of $E(n)$.

\begin{lemma}
Let $G_{0,n-1}$ be the group defined by generators $-1,e_1,\ldots,e_{n-1}$ and relations
$$(-1)^2=1,\qquad (e_i)^2=-1,\hspace{0.5cm}[e_i,e_j]=-1~(i\neq j),\qquad [e_i,-1]=1.$$
Then the groups $E(n)$ and $G_{0,n-1}$ are isomorphic.
\end{lemma}

\begin{proof}
We have: 
$$G_{0,n-1}=\{\pm e_{i_1}\ldots e_{i_{\ell}}\mid 1\leq i_j\leq n-1, e_i^2=-1, e_ie_j=-e_je_i\}.$$
Given an ordered set $I=(i_1,\ldots, i_m)$ with indices $i_j$ in $\{1, \ldots, n-1\}$, let $e_I=e_{i_1}\ldots e_{i_m}$.
Notice that if $I=(i_1,\ldots, i_m)$ and $J=(j_1,\ldots, j_p)$, then $e_Ie_J=e_{I\cup J}$,
where $I\cup J=(i_1,\ldots, i_m,j_1,\ldots, j_p)$. Now, define the map $\psi:G_{0,n-1}\to E(n)$ by
\begin{equation*}
\psi(e_I)=
\begin{cases}
e_I & |I|~{\text{even}}\\
e_{I\cup\{n\}} & |I|~{\text{odd}}
\end{cases}
\end{equation*}
First, we claim that $\psi(e_Ie_J)=\psi(e_I)\psi(e_J)$ for multi-indices $I$ and $J$.

\smallskip 

{\textbf{Case 1.}} $|I|$ and $|J|$ are both even. In this case, we have:
$$\psi(e_Ie_J)=\psi(e_{I\cup J})=e_{I\cup J}=e_Ie_J=\psi(e_I)\psi(e_J).$$

\smallskip 

{\textbf{Case 2.}} $|I|$ and $|J|$ are both odd. In this case, we have:
$$\psi(e_Ie_J)=\psi(e_{I\cup J})=e_{I\cup J}=e_Ie_J=e_Ie_J(-e_ne_n)=e_{I\cup\{n\}}e_{J\cup\{n\}}=\psi(e_I)\psi(e_J).$$

\smallskip 

{\textbf{Case 3.}} If $|I|$ is even and $|J|$ is odd, then
$$\psi(e_Ie_J)=\psi(e_{I\cup J})=e_{I\cup J\cup\{n\}}=e_Ie_{J\cup\{n\}}=\psi(e_I)\psi(e_J).$$

\smallskip 

{\textbf{Case 4.}} If $|I|$ is odd and $|J|$ is even, then
$$\psi(e_Ie_J)=\psi(e_{I\cup J})=e_{I\cup J\cup\{n\}}=e_{I\cup\{n\}}e_J=\psi(e_I)\psi(e_J).$$

\smallskip

Therefore, $\psi$ is a homomorphism. It is easy to see that $\psi$ is injective, and hence an isomorphism
since the groups $G_{0,n-1}$ and $E(n)$ have the same order.
\end{proof}

The groups $G_{0,n-1}$ have been classified by Salingaros \cite{Sal81,Sal82,Sal84} (cf. \cite{AVW18}).
We use this classification to write the group $E(n)\cong G_{0,n-1}$ as a central product.
This gives rise to the following list for fundamental groups of the principal orbits of the $G$-action on $M_{n\times n}(\R)$:\\
\begin{equation*}
E(n)\times\Z_2\cong
\begin{cases}
((Q_8)^{*\frac{n-4}{2}}*D_8)\times \Z_2^2 & n\equiv 0~({\text{mod}}~8)\vspace{0.1cm}\\
(Q_8)^{*\frac{n-1}{2}}\times\Z_2 & n\equiv 1, 3~({\text{mod}}~8)\vspace{0.1cm}\\
((Q_8)^{*\frac{n-2}{2}}*\Z_4)\times\Z_2 & n\equiv 2, 6~({\text{mod}}~8)\vspace{0.1cm}\\
(Q_8)^{*\frac{n-2}{2}}\times\Z_2^2 & n\equiv 4~({\text{mod}}~8)\vspace{0.1cm}\\
((Q_8)^{*\frac{n-3}{2}}*D_8)\times\Z_2 & n\equiv 5, 7~({\text{mod}}~8)
\end{cases}
\end{equation*}

We do not know, however, whether \emph{all} groups in Theorem \ref{main-thm:non-abelian part} do in fact arise 
as fundamental groups of principal leaves in a simply connected manifold.
\smallskip

\section{Virtually nilpotent fundamental group}\label{S:nilpotent fundamental group}

In this section, we consider singular Riemannian foliations $(M,\fol)$, where the fundamental group of $M$ is virtually nilpotent.
As the following example shows, the fundamental group of a principal leaf is not necessarily nilpotent in this case.

\begin{example}
Let $\hat{M}={\mathbb{C}}^2\times{\mathbb{S}}^1$ and consider the homogeneous foliation $\hat{\fol}$ 
on $\hat{M}$ induced by the linear action of $T^3=T^2\times S^1$. Let $M={\hat{M}}/{\Z_2}$, where 
the non-trivial element $g$ of $\Z_2$ acts by $g\cdot(z_1,z_2,t)=({\bar{z}}_1,{\bar{z}}_2,t+\frac{1}{2})$.
Note that $M$ inherits a singular Riemannian foliation $\fol={\hat{\fol}}/{\Z_2}$.
\par
The manifold $M$ is orientable, and is homotopy equivalent to ${\mathbb{S}}^1$. In particular, $M$ is nilpotent.
However, the principal leaf of $\fol$ is $T^3/{\Z_2}$ which has fundamental group
$$G=\Z^2\rtimes\Z=\langle a, b, c: cac^{-1}=a^{-1}, cbc^{-1}=b^{-1}, ab=ba\rangle.$$
Since $G_{\ell}=\langle a^{2^{\ell}}, b^{2^{\ell}}\rangle$ for any $\ell$, $G$ is not nilpotent.
\end{example}

Nevertheless, in what follows, we prove that the fundamental groups of the leaves contain a nilpotent subgroup of finite index. 

\begin{notation}
Throughout the rest of this section, $L_0$ denotes a principal leaf of $\fol$. Furthermore, we fix $p_0\in L_0$, 
and $K=\langle k_1,\ldots, k_m\rangle$ denotes the normal subgroup of $\pi_1(L_0,p_0)$ discussed at the beginning 
of Section \ref{S:fundamental group}. Recall that there is a homotopy fibration
$$L_0\overset{\iota_0}{\rightarrow}M_0\overset{\hat{\theta}}{\rightarrow}B.$$
which induces a long exact sequence
\[
0\to H\to \pi_1(L_0,p_0)\stackrel{(\iota_0)_*}{\to} \pi_1(M_0,p_0)\stackrel{\hat{\theta}_*}{\to} \pi_1(B,b)\to 1,
\]
where $H=\partial(\pi_2(B))$, as well as an action of $\pi_1(B,b)$ on $L_0$. Denote by $\hat{K}$ the group generated by $H$ 
and $c\cdot K$, for $c\in \pi_1(B,b)$. Notice that for every $\gamma\in \pi_1(M_0,p_0)$ with $c=\hat{\theta}_*(\gamma)$, 
and every $g\in\pi_1(L_0,p_0)$, $(\iota_0)_*(c\cdot g)=\gamma (\iota_0)_*(g)\gamma^{-1}$.
\end{notation}
%
%\begin{lemma}\label{lem:Z(K)}
%$[K,\pi_1(L_0)]\subseteq Z(K)$.
%\end{lemma}
%
%\begin{proof}
%Since $K$ is normal in $\pi_1(L_0)$, we have $[K,\pi_1(L_0)]\subseteq K$. Let $\gamma\in\pi_1(L_0)$ and $k\in K$.
%We use induction on the length of $k$ to prove that $[k,\gamma]\in Z(K)$. First, suppose that $k_i$ is a generator of $K$
%and consider the circle bundle ${\bf p}_i:L_0\to L_{\pi_i(p_i)}$ discussed in the proof of Theorem \ref{thm:nilpotent}. 
%Note that $\gamma k_i\gamma^{-1}=k_i^{\pm 1}$, depending on whether ${\bf p}_i$ is orientable or non-orientable
%along the curve representing $\gamma$. In particular, $[k_i,\gamma]$ is either $1$ or $k_i^2$ which is central 
%in $K$ by proof of Theorem A in \cite{GGR15}. Suppose now that $k=k_ik'$ for some generator $k_i$
%and some $k'\in K$, which satisfies $[k',\gamma]\in Z(K)$ by induction hypothesis. We have
%\begin{equation*}
%[k,\gamma]=k_ik'\gamma (k')^{-1}k_i^{-1}\gamma^{-1}=k_i[k',\gamma]\gamma k_i^{-1}\gamma^{-1}=[k',\gamma][k_i,\gamma]\in Z(K).\qedhere
%\end{equation*}
%\end{proof}
%
%The following proposition gives a sufficient and necessary condition for the fundamental group of a principal leaf to be nilpotent:
%
\begin{lemma}\label{lemma:central}
Let $(M,\fol)$ be a closed singular Riemannian foliation on a compact Riemannian manifold $M$.
If $\pi_1(M)$ is $n$-step nilpotent, then $(\pi_1(L_0,p_0))_{n+1}\subseteq \hat{K}$, where $(\pi_1(L_0,p_0))_{n+1}$ 
denotes the $(n+1)$-th group in the lower central series of $\pi_1(L_0,p_0)$.
\end{lemma}

\begin{proof}
Since removing strata of codimension $> 2$ does not change the fundamental group of $M$, we can assume that $M$ 
only contains singular strata of codimension $\leq 2$. In particular, we use the notation and results in Section 
\ref{SS:known-results}.

Letting $\iota:L_0\to M$ denote the inclusion, one then has
\[
\iota_*((\pi_1(L_0,p_0))_{n+1})\subseteq (\pi_1(M,p_0))_{n+1}=1.
\]
Therefore, given any curve $\alpha$ representing an element of $(\pi_1(L_0,p_0))_{n+1}$, there exists a disk 
$\bar{\iota}:\mathbb{D}^2\to M$ extending $\iota(\alpha)$. By transversality, this can be deformed to only intersect, 
transversely, the singular strata $\Sigma_1,\ldots,\Sigma_m$ of codimension 2, and the intersection consists of finitely many points $\{q_1,\ldots q_r\}$ with $q_j\in \Sigma_{i_j}$. For each $j=1,\ldots, r$, let $q'_j$ be a point in $\bar\iota(\mathbb{D}^2)$ 
close to $q_j$, let $u_j$ be a curve in $\bar\iota(\mathbb{D}^2)$ connecting $p_0$ to $q'_j$, and let $\psi_j$ a small loop 
in $\bar\iota(\mathbb{D}^2)$ based at $q'_j$, turning once around $q_j$. Finally, let $\gamma_j=u_j\star \psi_j\star u_j^{-1}$. Then:
\begin{enumerate}
\item For each $i=1,\ldots, r$, $[\gamma_j]\in \pi_1(M_0,p_0)$ is conjugate to $(\iota_0)_*(k_{i_j})$ 
with $k_{i_j}\in K\subseteq \pi_1(L_0,p_0)$. By the discussion before the proposition, it follows that $[\gamma_j]=(\iota_0)_*(c_j\cdot k_{i_j})$ for some $c_j\in \pi_1(B,b)$.
\item $(\iota_0)_*[\alpha]=[\gamma_1]\star\cdots\star[\gamma_r]=(\iota_0)_*((c_1\cdot k_{i_1})\star\cdots\star (c_r\cdot k_{i_r}))$ in $\pi_1(M_0,p_0)$.
\end{enumerate}

Since $H=\ker((\iota_0)_*)$, it follows that $[\alpha]=h((c_1\cdot k_{i_1})\star\cdots\star (c_r\cdot k_{i_r}))$ 
for some $h\in H$. In particular, $[\alpha]\in \hat{K}$, and therefore $(\pi_1(L_0,p_0))_{n+1}\subseteq \hat{K}$.
\end{proof}

We are finally ready to prove that if $(M, \fol)$ is a closed singular Riemannian foliation with $\pi_1(M)$ virtually nilpotent, then the fundamental group of every leaf is virtually nilpotent as well.

\begin{proof}[Proof of Theorem \ref{main-thm:virtually nilpotent}]
Notice that if $\pi:\hat{M}\to M$ is a finite cover, and $(\hat{M},\hat{\fol})$ is the lifted singular Riemannian foliation, one has that a leaf $\hat{L}\in \hat{\fol}$ has virtually nilpotent fundamental group if and only the corresponding leaf $\pi(\hat{L})\in \fol$ does. Therefore, up to replacing $M$ with a finite cover $\hat{M}$, we can assume that $\pi_1(M)$ is nilpotent.

Let $L_0$ be a principal leaf, and consider the Hurewicz homomorphism $h:\pi_1(L_0,p_0)\to H_1(L_0;\Z)$ and let $G=h^{-1}(2\cdot H_1(L_0;\Z))$. 
Clearly, $G$ has finite index in $\pi_1(L_0,p_0)$, Since $\pi_1(L_0,p_0)/G\cong H_1(L_0;\Z)/2\cdot H_1(L_0;\Z)$ 
is finite. We claim that if $\pi_1(M)$ is $n$-step nilpotent, then $G$ is $(n+1)$-step nilpotent.

By Lemma \ref{lemma:central}, $G_{n+1}\subseteq G\cap \hat{K}$. The proof is complete once we prove that $G$ 
commutes with $\hat{K}$. Notice that $\hat{K}$ is generated by $H$, and elements of the form $c\cdot k_i$ for $c\in \pi_1(B,b)$ 
and $k_i$ one of the generators of $K$. Recall that $H$ is central in $\pi_1(L_0,p_0)$ (in particular, $G$ commutes with $H$),
and for each $g\in \pi_1(L_0,p_0)$, $gk_ig^{-1}=k_i^{\pm 1}$. Since $\pi_1(B,b)$ acts on $\pi_1(L_0,p_0)$ 
by group automorphisms, it also follows that for every $g\in \pi_1(L_0,p_0)$, $g(c\cdot k_i)g^{-1}=(c\cdot k_i)^{\pm 1}$.

Notice that if $g(c\cdot k_i)g^{-1}=(c\cdot k_i)^\epsilon$ (for $\epsilon=\pm1$), then $g^{-1}(c\cdot k_i)g=(c\cdot k_i)^{\epsilon}$ as well. In particular, for every $g_1,g_2\in \pi_1(L_0,p_0)$, and every $(c\cdot k_i)\in \hat{K}$, one has:
\[
[g_1,g_2]\cdot (c\cdot k_i)[g_1,g_2]^{-1}=(c\cdot k_i).
\]
The main observation is that, by definition, any element $g\in G$ can be written as $g=g_3^2[g_1,g_2]\cdots [g_{2k-1},g_{2k}]$ 
for some $g_1,\ldots g_{2k}\in \pi_1(L_0,p_0)$ and therefore, for any generator $(c\cdot k_i)$ of $\hat{K}$, one has:
\begin{align*}
g(c\cdot k_i)g^{-1}=&g_3^2[g_1,g_2]\cdots [g_{2k-1},g_{2k}](c\cdot k_i)[g_{2k-1},g_{2k}]^{-1}\cdots [g_1,g_2]^{-1}g_3^{-2}\\
=&g_3^2(c\cdot k_i)g_3^{-2}=g_3(c\cdot k_i)^\epsilon g_3^{-1}\\
=&(c\cdot k_i)^{\epsilon^2}=(c\cdot k_i).
\end{align*}
Therefore, $G$ commutes with $\hat{K}$ and hence $G_{n+2}=[G,G_{n+1}]\subseteq [G,\hat{K}]=\{1\}$.

This proves that the principal leaves of $\fol$ have virtually nilpotent fundamental group. The corresponding statement for the non-principal leaves then follows from Lemma \ref{L:other-leaves}.
\end{proof}
%
%\todo{Added this part, it is the new version of Theorem C}
%Theorem \ref{main-thm:virtually nilpotent} follows as a straightforward corollary of  Proposition \ref{P:almost-thmC} above.
%\begin{proof}[Proof of Theorem \ref{main-thm:virtually nilpotent}]
%If $\hat{M}$ is a finite cover of $M$ with nilpotent fundamental group, and $(\hat{M},\hat{\fol})$ is the induced foliation, by Proposition \ref{P:almost-thmC} the principal leaves $\hat{L}$ of $\hat{\fol}$ have virtually nilpotent fundamental group. On the other hand, the covering $\hat{L}\to L$, with $L$ a principal leaf of $\fol$, induces an inclusion $\pi_1(\hat{L})\subseteq \pi_1(L)$ of finite index, thus the principal leaves $L$ of $\fol$ have virtually nilpotent fundamental group as well. Finally, for any leaf $L'$ of $\fol$, the foliated Slice Theorem \cite{MR19} implies that there is a fibration $L\to L'$ whose fiber $F$ has finitely many connected components. From the long exact sequence in homotopy one then has
%\[
%\pi_1(L)\to \pi_1(L')\to \pi_0(F)
%\]
%from which is follows that $\pi_1(L')$ is a finite extension of a quotient of $\pi_1(L)$, therefore it is virtually nilpotent as well.
%\end{proof}

\appendix

\section{Classification of quadratic forms over $\Z_2$}

The classification of quadratic forms over $\Z_2$ is well known. However, what appears usually in the literature is the classification of \emph{nondegenerate} quadratic forms, which is not what interests us here. Therefore, we provide the details of the classification.

\begin{proposition}\label{prop:quadratic form}
Every non-trivial quadratic form on $\Z_2^{\ell}$ is isometric to one of the following:
$$H_+^{\oplus m_1}\oplus 0^{m_2},\qquad H_-\oplus H_+^{\oplus m_1-1}\oplus 0^{m_2},\qquad H_+^{\oplus m_1}\oplus Q_1\oplus 0^{m_2-1},$$
where $2m_1+m_2=\ell$.
\end{proposition}

\begin{proof}
Let $H:\Z_2^2\times\Z_2^2\to\Z_2$ be the bilinear form given by
\[
H((x,y),(z,w))=xw+yz.
\] 
By the classification of bilinear forms over $\Z_2$ (cf. for example Proposition 1.8, Corollary 1.9 and the discussion below
in \cite{EKM08}), every symmetric bilinear form on a vector space $V$ over $\Z_2$ is isometric to $H^{m_1}\oplus 0^{m_2}$, where $2m_1+m_2=\ell$. By Equation \eqref{eq:quadratic form} in Section \ref{SS:quadratic}, 
it is easy to see that there are two equivalence classes of quadratic forms associated to $H^{m_1}$, 
that is the quadratic forms $Q=H_+^{m_1}$ and $Q=H_-\oplus H_+^{m_1-1}$, where $H_{\pm}:\Z_2^2\to\Z_2$ 
are given by 
\begin{align*}
& H_+(x,y)=xy,\\
& H_-(x,y)=x^2+y^2+xy.
\end{align*}
Similarly, corresponding to $0^{m_2}$, there are the quadratic forms $Q_0=0$ 
and $Q_{\alpha}(x_1,\ldots,x_{m_2})=\sum_{i=1}^{\alpha}x_i^2$ for any $1\leq \alpha\leq m_2$. 
However, $Q_{\alpha}$ is isometric to $Q_1\oplus 0^{m_2-1}$. Moreover, one has well known isometries
\begin{align*}
H_+^{\oplus m_1}\oplus Q_1\simeq H_-\oplus H_+^{\oplus m_1-1}\oplus Q_1,\qquad H_+^{\oplus 2}\simeq H_-^{\oplus 2},
\end{align*}
which conclude the proof.
\end{proof}

\bibliographystyle{plain}		% alpha   % plain

\end{document}